\documentclass[12pt,fleqn]{amsart}

\usepackage{amssymb}

\usepackage{color}
\usepackage{enumerate}
\newtheorem{theorem}{Theorem}[section]
\newtheorem{lemma}[theorem]{Lemma}

\newtheorem{problem}[theorem]{Problem}
\newtheorem{corollary}[theorem]{Corollary}

\theoremstyle{definition}
\newtheorem{definition}[theorem]{Definition}

\newtheorem{xca}[theorem]{Construction}

\newtheorem{axiom}[theorem]{Axiom}
\theoremstyle{remark}
\newtheorem{remark}[theorem]{Remark}
\numberwithin{equation}{section}

\begin{document}

\title[Properties of compacta that do not reflect]{On properties of compacta that do not reflect\\
in small continuous images}

\author[M.\ Magidor]{Menachem Magidor}
\address{Hebrew University of Jerusalem}
\email{mensara@savion.huji.ac.il}

\author[G.\ Plebanek]{Grzegorz Plebanek}
\address{Instytut Matematyczny, Uniwersytet Wroc\l awski}
\email{grzes@math.uni.wroc.pl}

\thanks{
This research started while the authors were visiting fellows at the Isaac Newton Institute for Math-
ematical Sciences, Cambridge, in the program "Mathematical, Foundational and Computational Aspects of the Higher Infinite (HIF)'' supported by EPSRC Grant Number EP/K032208/1.\\
The first author was also supported by the Simon Foundation.
The second author was partially supported by NCN grant 2013/11/B/ST1/03596 (2014-2017).}


\begin{abstract}
Assuming that there is a stationary set  of ordinals of countable cofinality in $\omega_2$ that does not reflect, we prove that there exists
a compact space which is not Corson compact and whose all continuous images of weight $\le\omega_1$ are Eberlein
compacta. This yields an example of a Banach space of density $\omega_2$ which is not weakly compactly generated but all
its subspaces of density $\le\omega_1$ are weakly compactly generated.

We also prove that under Martin's axiom countable functional tightness does not reflect in small continuous images of compacta.
\end{abstract}

\maketitle

\newcommand{\con}{\mathfrak c}
\newcommand{\eps}{\varepsilon}
\newcommand{\alg}{\mathfrak A}
\newcommand{\algb}{\mathfrak B}
\newcommand{\algc}{\mathfrak C}
\newcommand{\ma}{\mathfrak M}
\newcommand{\pa}{\mathfrak P}
\newcommand{\BB}{\protect{\mathcal B}}
\newcommand{\AAA}{\mathcal A}
\newcommand{\CC}{{\mathcal C}}
\newcommand{\cF}{{\mathcal F}}
\newcommand{\FF}{{\mathcal F}}
\newcommand{\GG}{{\mathcal G}}
\newcommand{\LL}{{\mathcal L}}
\newcommand{\NN}{{\mathcal N}}
\newcommand{\UU}{{\mathcal U}}
\newcommand{\VV}{{\mathcal V}}
\newcommand{\HH}{{\mathcal H}}
\newcommand{\DD}{{\mathcal D}}
\newcommand{\RR}{\protect{\mathcal R}}
\newcommand{\ide}{\mathcal N}
\newcommand{\btu}{\bigtriangleup}
\newcommand{\hra}{\hookrightarrow}
\newcommand{\ve}{\vee}
\newcommand{\we}{\cdot}
\newcommand{\de}{\protect{\rm{\; d}}}
\newcommand{\er}{\mathbb R}
\newcommand{\qu}{\mathbb Q}
\newcommand{\supp}{{\rm supp} }
\newcommand{\card}{{\rm card} }
\newcommand{\wn}{{\rm int} }
\newcommand{\ult}{{\rm ult}}
\newcommand{\vf}{\varphi}
\newcommand{\osc}{{\rm osc}}
\newcommand{\cov}{{\rm cov}}
\newcommand{\cf}{{\rm cf}}
\newcommand{\ol}{\overline}
\newcommand{\me}{\protect{\bf v}}
\newcommand{\ex}{\protect{\bf x}}
\newcommand{\stevo}{Todor\v{c}evi\'c}
\newcommand{\cc}{\protect{\mathfrak C}}
\newcommand{\scc}{\protect{\mathfrak C^*}}
\newcommand{\lra}{\longrightarrow}
\newcommand{\sm}{\setminus}
\newcommand{\uhr}{\upharpoonright}

\newcommand{\sub}{\subseteq}
\newcommand{\ms}{$(M^*)$}
\newcommand{\m}{$(M)$}
\newcommand{\MA}{$\protect{{\mathsf MA}}(\omega_1)$}
\newcommand{\clop}{\protect{\rm Clop} }
\newcommand{\fX}{\mathfrak X}
\newcommand{\fY}{\mathfrak Y}
\newcommand{\fZ}{\mathfrak Z}
\section{Introduction}

There is a significant amount of research related to properties of structures that reflect in substructures of smaller cardinality, see
e.g.\  Bagaria, Magidor, Sakai \cite{BMS}, Koszmider \cite{Ko99,Ko05}, Fuchino and Rinot \cite{FR11}, Tall \cite{Ta07}.
Reflection phenomena in topology are usually studied according to  the following pattern:

\begin{problem}\label{r1}
Does a topological space $X$ has a property (P) provided all its subspaces of small
cardinality have property (P)?
\end{problem}

Tall \cite{Ta07} surveys  results and problems of this type. Recently Tkachuk  \cite{Tk12} and Tkachuk and Tkachenko \cite{TT15} have investigated
which topological properties reflect in small continuous images which, in particular, amounts to asking the following kind of questions.

\begin{problem}\label{r2}
Does a topological space $X$ has property (P) provided  every continuous image of $X$ of weight $\le\omega_1$ has property (P)?
\end{problem}

Eberlein compacta and Corson compacta are two well-studied classes of compact spaces related to functional analysis, see the next section.
Answering two questions of type \ref{r2} posed in \cite{TT15}, we show in this note that it is relatively consistent that
neither  Eberlein compactness nor Corson compactness reflects in continuous images  of weight $\le\omega_1$. In fact, assuming that there is
a stationary set $S\sub\omega_2$ of ordinals of countable cofinality such that $S\cap\alpha$ is  stationary in no  $\alpha<\omega_2$, we construct a compact space $K$ of weight $\omega_2$ which simultaneously  answers in the negative
both the  questions: $K$  is  not Corson compact while all its
images of weight at most $\omega_1$ are Eberlein compacta that can be embedded into a Hilbert space.
In addition, our space $K$ gives partial negative answers to problems posed by  Jard\'on and Tkachuk (\cite{JT15}, Questions 4.13-15) on the reflection of type \ref{r1}
for Corson compacta and related classes.
The construction of the space is given in section 3 and uses the familiar idea of a ladder system associated to a given set $S\sub\omega_2$;
see, for instance, Pol \cite{Po79} and Ciesielski and Pol \cite{CP84} where a construction of this type was used to solve some problems on the structure of
$C(K)$ spaces. The existence of a stationary set in $\omega_2$ with the above-mentioned properties
is known to have an impact on other topological problems, see Fleissner \cite{Fl80}, 3.10.

In the framework of Banach space theory our result shows that it is relatively consistent that there exists a Banach space $X$ of density $\omega_2$ such that $X$ is not weakly compactly generated while every subspace $Y\sub X$ of density $\le\omega_1$ is weakly compactly compactly generated, see section \ref{wcg}.

In the final section of this note we  give a partial negative answer to another problem from \cite{TT15}. We  show, assuming a weak version of Martin's axiom, that countable functional tightness does not reflect in small continuous images of compact spaces.

We wish to thank Vladimir V.\ Tkachuk for sending us a preliminary version of \cite{TT15} and for his suggestion to  link  up our result with some questions asked in  \cite{JT15}.
We are also grateful to Mari\'an Fabian, Witold Marciszewski and Roman Pol for several valuable comments.

\section{Preliminaries}

All the topological spaces we consider here are assumed to be Hausdorff.
Given a topological space $X$, $w(X)$ denotes its topological weight, i.e.\
the minimal size of a base in $X$. Recall that a family  $\VV$ of nonempty open subsets of $X$ is
a {\em $\pi$-base} if every nonempty open set in $X$ contains some $V\in\VV$.

Our examples will be  constructed from some Boolean algebras.
If $\alg$ is a Boolean algebra then we write $\ult(\alg)$ for its Stone space (of all ultrafilters
on $\alg$). We write $\widehat{a}=\{x\in\ult(\alg): a\in x\}$ for $a\in\alg$. Recall that
sets $\widehat{a}$ form a base for the topology on $\ult(\alg)$.

We shall use the following result which is a very particular case of the Marde\v{s}i\'c factorization theorem \cite{Ma60}.
We enclose the sketch of a direct argument.

\begin{theorem}\label{mardesic}
Let $\alg$ be any Boolean algebra. If $L$ is a continuous image of $\ult(\alg)$ and $w(L)\le\omega_1$
then there is a subalgebra $\algb\sub \alg$ such that $|\algb|\le\omega_1$ and $L$ is a continuous image
of $\ult(\algb)$.
\end{theorem}

\begin{proof}
Note that $L$ has  a base $\UU$ of cardinality $\le\omega_1$ such that every $U\in\UU$ is $F_\sigma$.
Let $f:\ult(\alg)\to L$ be a continuous surjection. For every $U\in\UU$,  the set $f^{-1}(U)$ is of type $F_\sigma$
so it can be written as a union of countably many sets of the form $\widehat{a}$, $a\in\alg$.

It follows that there is $\algb\sub \alg$ of size at most $\omega_1$ such that, writing
$\pi:\ult(\alg)\to \ult(\algb)$ for the natural projection, we have $f(x)=f(y)$ whenever $x,y \in\ult(\alg)$ and $\pi(x)=\pi(y)$.
Hence we can write $f=f'\circ\pi$, where
$f':\ult(\algb)\to L$. It follows that $f'$  is continuous and the proof is complete.
\end{proof}

A compact space $K$ is said to be {\em Eberlein compact} if it is homeomorphic to a weakly compact subset
of some Banach space; equivalently, by the classical Amir-Lindenstrauss theorem, $K$ is Eberlein compact if it can be embedded into
\[ c_0(\kappa)=\{x\in\er^\kappa: \{\alpha: |x_\alpha|\ge\eps\}\mbox{ is finite for every } \eps>0\},\]
for some $\kappa$. Here $c_0(\kappa)$ is equipped with the topology inherited from $\er^\kappa$  (this topology agrees on bounded sets with the weak
topology of the Banach space $c_0(\kappa)$).

In particular, if $n\in\omega$ then every compact subset of
\[ \sigma_n(2^\kappa)=\{x\in 2^\kappa: |\{\alpha: |x_\alpha\neq 0\}|\le n\} ,\]
is Eberlein compact. In fact it is uniform Eberlein compact in the sense that it can be embedded as a weakly compact subspace of a Hilbert space
(note that $\sigma_n(2^\kappa)$ is a bounded subset of $l_2(\kappa)$).

A compact space $K$ is said to be {\em Corson compact} if there is $\kappa$ such that
$K$ is
homeomorphic to a subset of the $\Sigma$-product of real lines
\[ \Sigma(\er^\kappa) = \{x \in \er^\kappa: |\{\alpha : x_\alpha \neq 0\}| \le \omega\}.\]
Since $c_0(\kappa)\sub\Sigma(\er^\kappa)$, the class of Corson compacta contains (properly) the class of Eberlein compacta.
Negrepontis \cite{Ne84} and Kalenda \cite{Kal00}  offer extensive surveys on Eberlein and Corson compacta and related classes.
We only recall here that both uniform Eberlein compacta and Corson compacta  are stable   under  continuous images, see
e.g.\ \cite{Ne84}, 6.26 and \cite{Kal00}, p.\ 2.

A family $\FF$ in a Boolean algebra is said to be {\em centred} if $a_1\cap a_2\cap\ldots a_k\neq 0$ for every
natural number $k $ and every $a_i\in\FF$. We shall use the following standard fact.

\begin{lemma}\label{1:2}
For a Boolean algebra $\alg$ the following are equivalent:
\begin{enumerate}[(i)]
\item $\ult(\alg)$ is Corson compact;
\item there is a family $\GG\sub\alg$ generating $\alg$ and such that every
centred subfamily of $\GG$ is countable.
\end{enumerate}
\end{lemma}

\begin{proof}
$(i) \rightarrow (ii)$. Since $\ult(\alg)$ is Corson compact and zerodimensional,
$\ult(\alg)$ is homeomorphic to  a compact space $K$ contained in $\Sigma(2^\kappa)$
for some $\kappa$. The algebra of clopen subsets of $K$ is generated by
the family $\CC=\{C_\alpha: \alpha<\kappa\}$, where  $C_\alpha=\{x\in K: x_\alpha=1\}$.
Every centred subfamily of $\CC$ is  countable by the definition of $\Sigma(2^\kappa)$.

$(ii) \rightarrow (i)$. Take $f:\ult(\alg)\to 2^\GG$, where
$f(x)(G)=1$ if $G\in x$ and $=0$ otherwise. Then $f$ is continuous,
and $f[\ult(\alg)]\sub \Sigma(2^\GG)$ since
every ultrafilter on $\alg$ contains at most countably many generators from $\GG$. Moreover, $f$ is injective since $\GG$ generates $\alg$.
\end{proof}

\section{On Eberlein and  Corson compacta}

Let $\gamma$ be a limit ordinal. A set $F\sub\gamma$ is said to be {\em closed}  if it is closed in the interval topology defined on ordinals smaller that $\gamma$.
Such a set $F$ is unbounded in $\gamma$ if for every $\beta<\gamma$ there is $\alpha\in F$ such that $\beta<\alpha$. A set $S\sub \gamma$ is {\em stationary}
if $S\cap F\neq\emptyset$ for every closed and unbounded $F\sub\gamma$.

It is not difficult to check that the set $S_\omega=\{\alpha<\omega_2: \cf(\alpha)=\omega\}$ is stationary in $\omega_2$. However, such a set reflects in the sense
that, for instance, $S_\omega\cap\omega_1$ is stationary in $\omega_1$.
We shall work assuming the following.

\begin{axiom}\label{ax}
 There is a stationary set $S\sub\omega_2$ such that

\begin{enumerate}[(a)] \label{st:1}
\item $\cf(\alpha)=\omega$ for every $\alpha\in S$;
\item $S\cap\beta$ is not stationary in $\beta$ for every $\beta<\omega_2$ with $\cf(\beta)=\omega_1$.\label{st:1b}
\end{enumerate}
\end{axiom}

Note that in \ref{ax}(\ref{st:1b}) we can say that $S\cap\beta$ is not stationary in $\beta$ for every limit $\beta<\omega_2$ because if $\cf(\beta)=\omega$
then $\beta$ is a limit of a sequence of successor  ordinals.

Basic information on \ref{ax} can be found in Jech \cite{Jech}; recall that \ref{ax} follows from Jensen's principle $\square_{\omega_1}$ (\cite{Jech}, Lemma 23.6)
and hence it holds  in the constructible universe (\cite{Jech}, Theorem 27.1).
In fact one cannot  deny \ref{ax} and prove the consistency of the  the statement {\em every stationary set $S\sub\{\alpha<\omega_2: \cf(\alpha)=\omega\}$  reflects at some $\gamma<\omega_2$}
without assuming the existence of large cardinals,  see Magidor \cite{MM82} and \cite{Jech}, page 697.

\begin{xca}\label{construction}
Throughout this section we consider the space $K=\ult(\alg)$, where the Boolean algebra $\alg$ is defined as follows.

Fix a set $S\sub\omega_2$ as in \ref{ax}.
For every $\alpha\in S$ we pick an increasing sequence $(p_n(\alpha))_{n<\omega}$ of ordinals
such that $p_n(\alpha)\to \alpha$.
Put
\[A_\alpha=\{p_n(\alpha): n<\omega\},\quad
\mbox{and}\quad X=\bigcup_{\alpha\in S} A_\alpha.\]
Finally, let $\alg$ be the algebra of subsets of $X$ generated by finite subsets of $X$ together
with the family $\{A_\alpha: \alpha\in S\}$.
\end{xca}

We shall prove that $K=\ult(\alg)$ is not Corson compact because $S$ is stationary in $\omega_2$ while
the absence of stationary reflection for $S$ implies that $\ult(\algb)$ is  Eberlein compact  for every small subalgebra $\algb$ of $\alg$.

\begin{lemma}\label{st:3}
If $\alg$ is the algebra defined in  \ref{construction}
then the space $\ult(\alg)$ is not Corson compact
\end{lemma}

\begin{proof}
Suppose that there is a family $\GG\sub \alg$ as in Lemma \ref{1:2}(ii).
Note that every $A\in\alg$ is either countable or co-countable in $X$.
The family $\GG_0=\{G\in\GG: |X\sm G|\le\omega\}$ is centred so it is at most countable.  Hence, replacing every
$G\in\GG_0$ by its complement, we may assume that every $G\in\GG$ is countable.

Let $\GG_1=\{G\in\GG:|G|=\omega\}$. Note that every $G\in\GG_1$ is, modulo a  finite set,
a finite union of sets $A_\alpha$.

For every $\alpha\in S$ there must be $G_\alpha\in\GG_1$ such that $|A_\alpha\cap G_\alpha|=\omega$.
Indeed, otherwise $A_\alpha$ would be almost disjoint from every $G\in\GG_1$ so would not be in
the algebra generated by $\GG$.
Note that the function $\alpha\to G_\alpha$ is finite-to-one.

It follows that for every $\alpha\in S$ there is $\vf(\alpha)<\alpha$ such that $\vf(\alpha)\in G_\alpha$. By the pressing-down lemma, there is $\xi$ such that the set $\{\alpha\in S: \vf(\alpha)=\xi\}$ is stationary. It follows that $\{G\in\GG_1: \xi\in G\}$ is of cardinality $\omega_2$, and this is a contradiction.
\end{proof}

The second part of the argument is based on the following auxiliary result which is stated in a slightly stronger form suitable for
inductive argument.

\begin{lemma}\label{st:2}
For every  $\beta,\gamma$ such that $\beta<\gamma<\omega_2$ there is a family
\[\BB(\beta,\gamma)=\{B_\alpha: \alpha\in S\cap (\beta,\gamma) \}, \]
such that

\begin{enumerate}[(i)]
\item $B_\alpha\sub A_\alpha\sm\beta$ and $|A_\alpha\sm B_\alpha|<\omega$ for every $\alpha\in S\cap (\beta,\gamma)$;\label{w1}
\item $B_\alpha\cap B_{\alpha'}=\emptyset$ whenever $\alpha, \alpha'\in S\cap (\beta,\gamma)$ and $\alpha\neq\alpha'$.  \label{w2}
\end{enumerate}

\end{lemma}

\begin{proof}
We prove the assertion by induction on $\gamma$.

For the successor step $\gamma \to\gamma+1$ there is nothing to prove  in case $\gamma\notin S$.
Suppose $\gamma\in S$ and take $\BB(\beta,\gamma)$ satisfying (\ref{w1}) and (\ref{w2}).
Then $B_\alpha\cap A_\gamma$ is finite for every $\alpha<\gamma$, $\alpha\in S$ and therefore
\[\{B_\alpha\sm A_{\gamma}: \alpha\in (\beta,\gamma)\cap S\}\cup\{A_{\gamma}\sm\beta\},\]
is the required family for the interval $(\beta,\gamma)$.

Suppose that $\gamma$ is a limit ordinal. Then $S\cap\gamma$ is not stationary in $\gamma$ so there
is a closed unbounded set $C\sub \gamma$ such that $C\cap S=\emptyset$. In other words,
$S\cap\gamma$ is contained in a set $\gamma\sm C$ which is open and hence is a union of disjoint subintervals.

Fix $\beta<\gamma$.
If $\xi, \eta\in C$, $\beta<\xi<\eta$ and $(\xi,\eta)\cap C=\emptyset$ then we can apply the inductive assumption to $S\cap (\xi,\eta)$ and get the required family $\BB(\xi,\eta)$.
The union of families obtained in this way is clearly the family that satisfies (\ref{w1}) and (\ref{w2}).
\end{proof}

\begin{lemma}\label{st:4}
For every algebra $\algb\sub \alg$, where $\alg$ is as in \ref{construction}, if
 $|\algb|\le\omega_1$ then the space $\ult(\algb)$ is uniform Eberlein compact.
\end{lemma}

\begin{proof}
Let $\gamma<\omega_2$ and let $\algb_\gamma$ be a subalgebra of $\alg$ generated by
all finite sets in $X$ and the family $\{A_\alpha:\alpha<\gamma\}$. It follows directly from
Lemma \ref{st:2} that $\algb_\gamma$ has a generating family $\GG$ such that there are no three
different elements in $\GG$ having nonempty intersection. Then the space
$\ult(\algb_\gamma)$ can be embedded into $\sigma_2(2^\gamma)$ (as in  Lemma \ref{1:2}) so
it is uniform Eberlein compact.

Now every subalgebra $\alg\sub\alg$ of size $\le\omega_1$ is included in $\algb_\gamma$
for some $\gamma<\omega_2$. Hence $\ult(\algb)$ is a continuous image of $\ult(\algb_\gamma)$
and thus it is uniform Eberlein compact as well.
\end{proof}

The following answers simultaneously, subject to our set-theoretic assumption ,
Questions 4 and 5 in \cite{TT15}.

\begin{theorem}\label{st:5}
Assume \ref{ax}. There is a scattered compact space $K$ with the third derivative empty such that

\begin{enumerate}[(i)]
\item $K$ is not Corson compact (in fact it is not $\omega_2$-Corson compact in the sense
of \cite{Kal00});
\item If $L$ is a continuous image of $K$ and $w(L)\le\omega_1$ then $L$ is uniform Eberlein compact.
\end{enumerate}

\end{theorem}

\begin{proof}
We take $K=\ult(\alg)$, where $\alg$ is the algebra defined above in \ref{construction}.
Since $K$ is a Stone space of an algebra generated by an almost disjoint family, it is clear that
$K^{(3)}=\emptyset$. Indeed, every ultrafilter $x\in\ult(\alg)$ is either  principal or there is a unique $\alpha$ such that
$A_\alpha\in X$ or else $x\in K^{(2)}$ is the unique ultrafilter containing all $X\sm A_\alpha$.

Then $K$ is not Corson compact by Lemma \ref{st:3}. If $L$ is a continuous images of $K$
and $w(L)\le\omega_1$ then $L$ is uniform Eberlein compact by Theorem \ref{mardesic}, Lemma \ref{st:4} and
the fact that uniform Eberlein compacta are stable under continuous images.
\end{proof}

\begin{remark} In connection with Theorem \ref{st:5}, it is worth remarking that since the space $K$ used in its proof is scattered, every continuous image of
$K$ is zerodimensional so in fact the use of \ref{mardesic} is not essential. Moreover, the following results are closely related to the final conclusion on
continuous images of $K$ having  small weight.

\begin{enumerate}
\item Alster \cite{Al79} proved that every scattered Corson compact space is Eberlein compact.
\item Bell and Marciszewski \cite{BM07} proved that a scattered Eberlein compact space of height at most
$\omega+1$ is uniform Eberlein compact.
\end{enumerate}
\end{remark}

As we mentioned in the introduction, the space $K$ from Theorem \ref{st:5} settles in the negative some reflection problems of
type \ref{r1}.

\begin{lemma}\label{st:6}
Let (P) be a property of compact space that is stable under taking closed subspaces. If $K$ is a compact space and all continuous images
of weight $\le \omega_1$ have property (P) then all closed subsets $L$ of $K$ of cardinality $\le\omega_1$ have property (P).
\end{lemma}

\begin{proof}
Take a closed subspace $L\sub K$ with $|L|\le\omega_1$. Then there is a family $\cF$ of continuous functions $K\to [0,1]$ such that
$|\cF|\le\omega_1$ and $\cF$ separates  the  points of $L$. Let  $g:K\to [0,1]^\cF$ be the diagonal mapping, i.e.\  $g(x)(f)=f(x)$ for $f\in\cF$.
Then $\widetilde K =g[K]\sub [0,1]^\cF$ so $w(\widetilde K)\le |\cF|\le\omega_1$ and hence $\widetilde K$ has property (P). It follows that
$\widetilde L =g[L] \sub\widetilde K$ also has property (P), and $\widetilde L$ is homeomorphic to $L$.
\end{proof}

\begin{corollary}\label{st:7}
Assume \ref{ax} and take the space $K$ as in Theorem \ref{st:5}. Then $K$ is not Corson compact while for every
$Y\sub K$, if $|Y|\le\omega_1$ then $\ol{Y}$ is uniform Eberlein compact.
\end{corollary}

\begin{proof}
Recall that $X$ was defined as the union of all the sets $A_\alpha$, $\alpha\in S$.
If $Y\sub X$ and $|Y|\le\omega_1$ then $Y\sub X\cap\gamma$ for some $\gamma<\omega_2$ and this easily implies that $|\ol{Y}|\le\omega_1$.
If $Y\sub K^{(1)}$  then $\ol{Y}\sub Y\cup \{\infty\}$, where $\infty$ is the only point in $K^{(2)}$.

We conclude that  $|\ol{Y}|\le\omega_1$ for every $Y\sub K$ with $|Y|\le\omega_1$ and
the assertion follows from Lemma \ref{st:6}.
\end{proof}

 The corollary above gives partial negative answers to problems posed by Jard\'on and Tkachuk (\cite{JT15}, Questions 4.13-15) if we
 assume the continuum hypothesis together with \ref{ax}, so for instance if we are in the constructible universe.

\section{On WCG Banach spaces}\label{wcg}

A Banach space $\fX$ is said to be {\em weakly compactly generated} (WCG) if there is a weakly compact set $C\sub \fX$ such that
$\fX=\overline{\rm span}(C)$, see  Negrepontis \cite{Ne84} for a survey on WCG Banach spaces and further references.
 
Given a compact space $K$,   $C(K)$ denotes the Banach space of real-valued continuous
functions on $K$ (equipped with the usual supremum norm). By a classical result due to Amir and Lindenstrauss,
there is a natural duality between the class of WCG Banach spaces and the class of Eberlein compacta.
In particular, the following holds (see e.g.\ \cite{Ne84}, Theorem 6.9).

\begin{theorem} \label{wcg:0}
Given a compact space $K$, the Banach space $C(K)$ is WCG if and only if $K$ is Eberlein compact.
\end{theorem}

Let us recall that the class of $WCG$ Banach spaces in not stable under taking subspaces.
We shall use the following fact which is a very particular case of  a result due to Fabian \cite{Fa87}.

\begin{theorem}[Fabian]\label{wcg:1}
Let $K$ be a compact scattered space. If $\fZ$ is a WCG subspace of $C(K)$ then every subspace $\fY$ of $\fZ$ is WCG too.
\end{theorem}

In the spirit of reflection problems considered above, one can wonder if,  for a Banach space $\fX$, the property of being WCG reflects in subspaces of small density.
The following reformulation of  Theorem \ref{st:5} gives a partial negative answer to such a question.

\begin{theorem}\label{wcg:2}
Under \ref{ax} there exists a Banach space $\fX$ of density $\omega_2$ such that

\begin{enumerate}[(i)]
\item $\fX$ is not weakly compactly generated;

\item every subspace $\fY\sub \fX$ of density $\le\omega_1$ is weakly compactly generated.
\end{enumerate}
\end{theorem}

\begin{proof}
We take the space $K$ as if Theorem \ref{st:5} and consider $\fX=C(K)$. Then $\fX$ is not WCG by Theorem \ref{wcg:0}, since $K$ is not Eberlein compact.
As the weight of $K$ is $\omega_2$, the space $\fX$ is of density $\omega_2$.

Take any subspace $\fY\sub \fX$ of density $\le\omega_1$. Then there is a family $\FF\sub Y$ of size $\le\omega_1$ which is dense in $\fY$.
Let $\theta:K\to\er^\FF$ be the diagonal mapping, that is $\theta(x)(f)=f(x)$ for $x\in K$ and $f\in\FF$. Then $L=\theta[K]$ is a continuous image of $K$
of weight $\le\omega_1$ so $L$ is Eberlein compact  by Theorem \ref{st:5}(ii) and the Banach space $C(L)$ is WCG by Theorem \ref{wcg:0}.

The space $\fZ=\{h\circ\theta: h\in C(L)\}$ is an isometric copy of $C(L)$ so $\fZ$ is WCG too. Now $\fY$  is clearly a subspace of $\fZ$ so $\fY$ is WCG by
Theorem \ref{wcg:1}, and the proof is complete.
\end{proof}

\section{On countable functional tightness}

\begin{definition}\label{ft:1}
For a topological space $X$ and a cardinal number $\kappa$ we write $t_0(X)\le\kappa$
if  every function $f:X\to\er$ is continuous provided
$f_{|Y}:Y\to\er$ is continuous for every  subspace $Y\sub X$ with $|Y|\le\kappa$.
The corresponding cardinal number $t_0(X)$ is called
the {\em functional tightness} of the space $X$.
\end{definition}

Recall that $t(X)$, the tightness of a space $X$ is defined so that for every $A\sub X$ and every $x\in\overline{A}$ there is
$B\sub A$ such that $|B|\le t(X)$ and $x\in \overline{B}$.
The following fact can be found in  \cite{Ar83}.

\begin{lemma}\label{ft:2}
The functional tightness $t_0(X)$ does not exceed the density of $X$ for every  space $X$.
Moreover, $t_0(X)\le t (X)$.
\end{lemma}

Tkachuk (see Theorem 2.11 in \cite{Tk12}) proved that if $K$ is a compact space of uncountable
tightness then $K$  has a continuous image $L$ of uncountable tightness with $w(L)=\omega_1$.

Recall that $t_0(2^\kappa)=\omega$ if and only if there are no measurable cardinals
$\le\kappa$, see Uspenskii \cite{Us83}, cf.\ \cite{Pl91}. Using this theorem
it is noted in \cite{TT15} that if there are measurable cardinals then
the countable functional tightness does  not reflect in small continuous images of compacta.

Let $\lambda$ be the Lebesgue measure on $[0,1]$. Write $\NN$ for the ideal of $\lambda$-null
sets. Recall that the assertion $\cov(\NN)>\omega_1$ means that $[0,1]$ cannot be covered
by $\omega_1$-many sets from $\NN$.

We shall work in the measure algebra $\alg$ of the Lebesgue measure on $[0,1]$;
the corresponding measure on $\alg$ is still denoted by $\lambda$. The following is an immediate
consequence of a result due to Kamburelis \cite{Ka89}, Lemma 3.1; see also
\cite{BNP15}, Theorem 4.4.

\begin{theorem}\label{ft:2.5}
If $\cov(\NN)>\omega_1$ then every continuous image of $\ult(\alg)$ of $\pi$-weight $\le\omega_1$
is separable.
\end{theorem}

We can now give   a (partial) negative solution to Question 4.3 from \cite{TT15}.

\begin{theorem}\label{ft:3}
Assuming $\cov(\NN)>\omega_1$, there is a compact space $S$ with $t_0(S)>\omega$, such that
$t_0(L)=\omega$ for every continuous image $L$ of $S$ of weight $\omega_1$.
\end{theorem}

Our result is based on the construction described in the following lemma.

\begin{lemma}\label{ft:4}
Let $(s_n)_{n}$ be a pairwise disjoint sequence in $\alg^+$. Let
\[\FF=\{a\in\alg: \lim_n \lambda(a\cap s_n)/\lambda(s_n)=1\},\]
 \[F=\{x\in\ult(\alg): \FF\sub x\}.\]
Then

\begin{enumerate}[(i)]
\item  $\FF$ is a non-principal filter in $\alg$;
\item $F$ is a closed subset of $\ult(\alg)$ with empty interior;
\item for every countable $Y\sub \ult(\alg)\sm F$ we have $\overline{Y}\cap F=\emptyset$.\label{w2b}
\end{enumerate}
\end{lemma}

\begin{proof}
Part (i) follows by standard calculations and part (ii) is a direct consequence of (i). We shall check (iii). Let $Y=\{y_n: n\in\omega\}\sub\ult(\alg)\sm F$. For every $n$ we have $y_n\notin F$ so there
is $a^n_0\in y_n $ such that $-a^n_0\in \FF$. Then we choose
a decreasing sequence $(a^n_k)_k$ such that

\begin{enumerate}[(a)]
\item $a^n_k\leq a^n_0$, $a^n_k\in y_n$ for every $k$
\item $\lim_k\lambda (a^n_k)=0$.
\end{enumerate}

The following fact can be proved by a standard diagonalization (cf.\ \cite{Ku75}).
\medskip

{\sc Claim.} There is a function $g:\omega\to\omega$ such that writing
$ a_g:=\bigcup_{n\in\omega}a^n_{g(n)}$,
we have $\widehat{a_g}\cap F=\emptyset$.
\medskip

Using Claim we get $Y\sub \widehat{a_g}$ and it follows that $\overline{Y}\cap F=\emptyset$.
\end{proof}

\begin{proof} (of Theorem \ref{ft:3}).
Let $S$ be the Stone space of the measure algebra $\alg$. Take the set $F\sub S$
from Lemma \ref{ft:4}. Then condition (\ref{w2b}) implies that the function $\chi_F:S\to\er$
is continuous on every countable subspace of $S$. But $\chi_F$ is clearly not continuous
because the interior of $F$ is empty. Hence $t_0(S)>0$.

Let now $L$ be a continuous image of $S$ such that $w(L)\le\omega_1$.
Then $L$  is separable by Theorem \ref{ft:2.5} and $t_0(L)=\omega$ by Lemma \ref{ft:2}, so the proof is complete.
\end{proof}

\begin{remark} We enclose some remarks concerning Theorem \ref{ft:3}

\begin{enumerate}
\item Lemma \ref{ft:3} originates in Kunen \cite{Ku75}; see Plebanek \cite{Pl98}
for other applications.
\item In fact, under MA$(\omega_1)$ one can check that in the proof we actually get $t_0(S)>\omega_1$,
since under MA$(\omega_1)$ one can strengthen (\ref{w2b}) of Lemma \ref{ft:4} to saying that
$\overline{Y}\cap F=\emptyset$ for every $Y\sub S\sm F$ with $|Y|\le\omega_1$.
\item The proof of \ref{ft:3} says a bit more,  that $t_0(L)=\omega$ whenever $L$ is a continuous image of $S$ having
a $\pi$-base of cardinality $\le\omega_1$.
\end{enumerate}

\end{remark}

\end{document}